\documentclass{article}
\usepackage[english]{babel}
\usepackage[letterpaper,top=2cm,bottom=2cm,left=3cm,right=3cm,marginparwidth=1.75cm]{geometry}
\usepackage{amsmath,amsfonts,amsthm,authblk,mathtools}
\usepackage{graphicx}
\usepackage[colorlinks=true, allcolors=blue]{hyperref}
\newcommand{\p}{\partial}
\newcommand{\R}{\mathbb{R}}
\newtheorem{theorem}{Theorem}[section]
\newtheorem{lemma}{Lemma}[section]
\newtheorem{prop}{Proposition}[section]
\newtheorem{definition}{Definition}[section]
\newtheorem{remark}{Remark}[section]
\newcommand*\diff{\mathop{}\!\mathrm{d}}
\newtheorem*{remark*}{Remark}

\usepackage{hyperref}

\title{Phase separation and morphology formation in interacting ternary
mixtures under evaporation: Well-posedness and numerical
simulation of a non-local evolution system}
\author[$\ddagger$]{Rainey Lyons\footnote{rainey.lyons@kau.se}}

\author[$\dagger$]{Emilio N. M. Cirillo} 
\author[$\ddagger$]{Adrian Muntean}

\affil[$\ddagger$]{Department of Mathematics and Computer Science, Karlstad University, Sweden
 }

\affil[$\dagger$]{Department of Basic and Applied Sciences for Engineering (SBAI), Sapienza University of Rome, Italy}

 \usepackage{color}

\begin{document}

\maketitle
\begin{abstract}
We study a nonlinear coupled parabolic system with non-local drift terms modeling at the continuum level the inter-species interaction within a ternary mixture that allows the evaporation of one of the species. In the absence of evaporation, the proposed system coincides with the hydrodynamic limit of a stochastic interacting particle system of Blume--Capel--type driven by the Kawasaki dynamics. Similar governing dynamics are found in models used to study morphology formation in the design of organic solar cells, thin adhesive bands, and other applications.  
We investigate the well-posedness of the target system and present preliminary numerical simulations which incorporate `from the top' evaporation into the model. We employ a finite volumes scheme to construct approximations of the weak solution and illustrate how the evaporation process can affect the shape and connectivity of the evolving-in-time morphologies.

\vskip0.25cm
{\bf Keywords:} Phase separation, coupled non-local parabolic system, ternary mixture, evaporation,  well-posedness, weak solutions, numerical simulation
\vskip 0.25cm
{\bf AMS Subject classifications:} 	 35K55, 35Q70, 65N08  

\end{abstract}

\section{Introduction}
In this paper, we study the solvability and the numerical simulation of a nonlinear coupled parabolic system with a specific shape of the non-linear non-local drift, which is derived in \cite{Marra} as hydrodynamic limit of the Blume-Capel model endowed with Kawasaki dynamics describing the interaction of ternary mixture of particles. Such mixtures typically  consist of two different solutes mixed within a solvent and allow for phase separation. The situation is rather typical in modern materials science when thin films are involved; see e.g., Section \ref{applications} for a discussion of the application of such scenario to producing morphologies as needed for organic solar cells. In the applications discussed below, the solvent particle evaporates and plays a crucial role in the morphologies produced by the separated phases. The presence of such a non--equilibrium process makes this scenario different compared to the more classical phase separation settings arising in crystal growth, metallurgy, and so on, normally treated by Allen--Cahn--type  or Cahn--Hilliard--type equations. Specifically, by controlling the evaporation mechanism a local equilibrium is frozen, i.e., one has the possibility to actively select specific morphologies as the end configuration. We refer the reader to the recent works \cite{Andrea_EPJ,Andrea_PhysRevE, Mario}, where this setting has been explored using Monte Carlo simulations for suitable lattice--based models. Although we are driven by these two concrete applications, our work can be seen as a continuation of the efforts started earlier by Lebowitz and Giacomin \cite{Giacomin,giacomin1998phase}. 

For a good grip on the design of morphologies out of {\em ad hoc} polymer-polymer-solvent interactions, it is crucial to have sufficiently rich models posed at the continuum level that inherit, from the interacting particle systems, the ability to produce physically meaningful phase separation.  This would set the stage for process and shape optimization operations to reach optimal macroscopic transport and reaction properties, mediated by optimal morphology shapes. However, in spite of huge efforts in the statistical mechanics and applied mathematics communities, rigorous derivations of the corresponding hydrodynamic limit equations for mixtures of interacting particle systems with evaporation are currently out of reach. The main technical difficulties in performing rigorously the hydrodynamic limit process seem to arise because of the evaporation component of the process. To allow for the evaporation of the solvent, we modify the limit system with a linear evaporation term to mimic the release of the solvent. 
The problem setting is relevant, with a rich phenomenology, 
for the presence of three coexisting phases undergoing 
the separation process.

\subsection{A model without evaporation} 

Within this framework, the main interest is on the solvability of  the following system of two coupled non-local nonlinear parabolic equations 
\begin{equation}\label{Eq:MainModel_withoutEvap}
\left\lbrace
\begin{split}
    \p_t m &= \nabla \cdot \left[\nabla m - 2 \beta (\phi -m^2 ) (\nabla J * m) \right], \quad (t,x) \in (0,T)\times\Omega =: Q_T,\\
    \p_t \phi &= \nabla \cdot \left[ \nabla \phi - 2 \beta m (1 - \phi) (\nabla J * m) \right].
\end{split}\right. 
\end{equation}
Here, $t$ and $x$ represent the time and space variable, respectively, $m$ represents the average spin density (also called magnetization), and $\phi$ represents the solute volume concentration (we expand upon the meaning of these terms in a later section). Additionally, $\Omega \, \subset \R^d $ is a cube with spatially periodic boundary conditions (i.e., homeomorphic to the unit torus), $T$ is some positive time, $\beta >0$ is a constant representing the inverse temperature,
$J\in C_+^2(\bar \Omega)$ is a symmetric compactly supported potential such that $\int_{\Omega} J(r)dr=1$. We note that in the analysis to follow, such regularity on $J$ is unnecessary. Indeed, the proofs below readily hold for $J \in W^{1,\infty}(\R^d)$ and small adjustments can be made for to include $J \in W^{1,2}(\R^d)$ and possibly $J\in W^{1,1}(\R^d)$; however, we maintain the $C^2$ regularity to stay true to the Kac interpretation of the potential. For the system  \eqref{Eq:MainModel} we prescribe the initial data
\begin{equation}
m(t=0)=m_0  \mbox{ and } \phi(t=0)=\phi_0 \mbox{ in } \bar\Omega.
\end{equation} Additionally, we impose the following physically motivated assumption which we discuss in the next section,
\begin{equation}\label{Assump:MPhi}
    0 \leq |m_0| \leq \phi_0 \leq 1 \text{ almost everywhere in } \bar{\Omega}.
\end{equation}

 The structure of system \eqref{Eq:MainModel} (posed for $\Omega = \R^d$) is derived rigorously in \cite{Marra} as the hydrodynamic limit of the Kawasaki dynamics for the Blume--Capel model with range of 
interaction $\gamma^{-1}$, magnetic field $h_1$, and chemical potential
$h_2$,
whose Hamiltonian 
in a finite square $V\subset\mathbb{Z}^d$ 
is 
\begin{equation}\label{Eq_HamMarra}
H_\gamma (\sigma) 
= 
\frac{1}{2} \sum_{\substack{x \neq x' \in V}} 
J_\gamma (x-x')[\sigma(x)-\sigma(x')]^2 
-\sum_{x\in V}h_1 \sigma(x) 
-\sum_{x\in V}h_2 \sigma^2(x), 
\end{equation}
where $J_\gamma: \mathbb{R}^d\to \mathbb{R}$ 
is a Kac potential function, i.e.,
\begin{equation}
\label{geigamma}
J_\gamma(r)=\gamma^d J(\gamma r)
\end{equation}
for all $r\in \mathbb{R}^d$ and $\sigma(x) \in \{-1,0,+1\}$ is the spin of a particle at position $x \in V$. We discuss the intuition on the discrete spin model below, but also refer the reader to the monograph \cite{Presutti} for more information on the context.

As discussed in \cite{Marra}, letting $\vec{u} := (m,\phi)$, this system can also be written as a gradient flow structure given by 
\begin{equation}\label{Eq:GradStruct}
    \p_t \Vec{u} = \nabla \cdot \left( M\nabla \frac{\delta \mathcal{F}}{\delta \vec{u}}\right),
\end{equation}
where the mobility, \[M := \beta (1-\phi) \begin{bmatrix}
    \phi + \frac{\phi^2 - m^2}{1-\phi} & m \\
    m & \phi
\end{bmatrix}\]
and the free energy functional, $\mathcal{F}$, is given by
\[\mathcal{F}(\Vec{u}) = \int_\Omega f(\vec{u}) \diff{x} + \frac{1}{2} \int_\Omega \int_\Omega J(x-x') [m(x) - m(x')]^2 \diff{x'} \diff{x}, \]
where 
\[f(\Vec{u}) := \phi - m^2 + \beta^{-1} \left[ \frac{1}{2}(\phi+m)\log(\phi+m) + \frac{1}{2}(\phi-m)\log(\phi-m) + (1-\phi) \log(1-\phi) - \phi \log(2)\right].\]
Also discussed in \cite{Marra} is that $\mathcal{F}$ is a Lyapunov functional for \eqref{Eq:GradStruct} and that 
\[\frac{\diff{}}{\diff{t}} \mathcal{F} = - \sum_{i = 1}^{d} \sum_{\alpha, \beta = 1}^2 \int_\Omega \p_i \frac{\delta \mathcal{F}}{\delta \vec{u}_\alpha} \, M_{\alpha,\beta} \, \p_i \frac{\delta \mathcal{F}}{\delta \vec{u}_\beta}. \]
Unfortunately, such a variational structure is lost when one makes adjustments to include evaporation.

The first equation in system \eqref{Eq:MainModel} is similar to the non-local Cahn-Hilliard equation well studied in the literature \cite{giacomin1998phase,gal2017nonlocal}. In fact, it is easily seen for $\phi \equiv 1$ (i.e., a two species system), the model presented here reduces to the non-local Cahn-Hilliard equation with degenerate mobility and logarithmic free energy. However, the coupling of the non-linearity to the second equation, allows for a mechanism to track the solvent ratio in the mixture. Additionally, there is hidden structure in the equations (which we see later in system \eqref{Eq:SumDiff_System}) that is reminiscent of equations in various applications including the McKean--Vlasov equations studied in swarms \cite{carrillo2014derivation} and the Fokker--Planck equation studied in mean field games \cite{lasry2007mean} and opinion dynamics \cite{CHAZELLE2017365,carrillo2020long}. The main difference between these examples and the equations in system \eqref{Eq:SumDiff_System}, is the non-linearity present in the drift term. While this structure is not surprising to find, as the previously mentioned McKean--Vlasov equations have been used to study interacting particle systems in different applications (e.g., \cite{liu2021long}), this particular form of the non--linearity in the drift appears to be quite novel.

\subsection{Motivating applications and microscopic interpretation}\label{applications}
Phase separation models, such as the Cahn-Hilliard equation, have seen a wide variety of applications (see, e.g., \cite{Miranville}). The choice of the particular model \eqref{Eq:MainModel} is motivated by several applications such as the formation of rubber--based zones related to the design of thin adhesive bands \cite{creton2016rubber,Muller_etal2022}; or the formation of morphologies in organic solar cells \cite{hoppe2004organic}. In both of these applications, a mixture consisting of three species of particles (two solutes and one solvent) undergoes phase separation to form the aforementioned morphologies/zones. In both applications, evaporation of the solvent plays a key role as the ratio of solvent to the mixture has been shown to affect the resulting structure of the morphologies (see, e.g.,\cite{Andrea_PhysRevE} for a discrete simulation and \cite{LyonsMunteanetal2023} for simulations on system \eqref{Eq:MainModel}). 

Modeling the evaporation process in conjunction with phase separation is usually done at the discrete level with a Monte--Carlo simulation with the solvent particles taking some particular behavior (e.g.,point--wise replacement or replacement at the boundary \cite{Andrea_EPJ,Andrea_PhysRevE,Mario}). The choice of solvent behavior and spatial dimension changes the interpretation of this evaporation process. Generally for two dimensional simulations, one takes either a `from the top' or a `from the side' perspective. The main difficulty is finding the continuum limit of such evaporation models. Therefore, one of the goals of this manuscript is to find a way to adjust system \eqref{Eq:MainModel}, a system derived without an evaporation process, to reasonably simulate evaporation with either perspective.  

Discrete lattice simulations of this behavior often assign a `spin' to the particles denoted as in \eqref{Eq_HamMarra} by $\sigma (x) \in \{-1,0,+1\}$. In this set-up, the solvent particle would be assigned spin $0$, while the solute particles would have spin $+ 1$ or $-1$. The concentration of the solution at a site $x$ would then be represented by $|\sigma(x)|$ (traditionally in the literature as $\sigma(x) ^2$). This is the measure of how much solute (independent of the spin direction) is located at a particular site. 

Relating this intuition back to system \eqref{Eq:MainModel}, we let $m$ and $\phi$ be continuous measures of the average spin (also referred to as magnetization in \cite{Marra}) and concentration in a given region, respectively. In the context of the discrete model, these terms mean for a collection of sites $A := \{x_i\}_{i=1}^{N}$ 
\[m_A = \frac{1}{N} \sum_{i=1}^N \sigma (x_i) \text{ and } \phi_A = \frac{1}{N} \sum_{i=1}^N \sigma(x_i)^2 = \frac{1}{N} \sum_{i=1}^N |\sigma(x_i)|. \]
Therefore at this level, inequalities \eqref{Assump:MPhi} and \eqref{Ineq:MPhi} are just applications of the triangle inequality. Since system \eqref{Eq:MainModel} is the continuum limit of this stochastic interpretation under the Hamiltonian \eqref{Eq_HamMarra}, we maintain this inequality with assumption \eqref{Assump:MPhi}.

Returning back to the continuous setting and system \eqref{Eq:MainModel}, $m: \bar Q_T \longrightarrow \R$ is a function describing the net spin of particles in a given region. In other words, for a given Borel set $A \subseteq \Omega$, $\int_A m(t,x) \diff x$ is the net spin of particles in the set $A$ at time $t$. As the spin in the Hamiltonian \eqref{Eq_HamMarra} only takes values in $\{-1,0,+1\}$, we aim to show that $|m| \leq 1$ almost everywhere in $Q_T$. To capture the nuances of solvent particles, the function $\phi: \bar Q_T \longrightarrow \R$ represents the concentration of solute particles, i.e.,$\int_A (1-\phi(t,x))\diff x$ represents the solvent ratio in the set $A$ at time $t$. Since $\phi$ represents the concentration, we will show that the physical inequality $0\leq |m| \leq \phi \leq 1$ holds almost everywhere in $Q_T$.

\subsection{A model with evaporation}
Keeping all previous assumptions, we now adjust system \eqref{Eq:MainModel_withoutEvap} to include a source term for the concentration component.
\begin{equation}\label{Eq:MainModel}
\left\lbrace
\begin{split}
    \p_t m &= \nabla \cdot \left[\nabla m - 2 \beta (\phi -m^2 ) (\nabla J * m) \right], \quad (t,x) \in (0,T)\times\Omega =: Q_T,\\
    \p_t \phi &= \nabla \cdot \left[ \nabla \phi - 2 \beta m (1 - \phi) (\nabla J * m) \right] + F(\phi).
\end{split}\right. 
\end{equation}
Here, $F:[0,1] \longrightarrow \R$ is a bounded non-increasing Lipschitz continuous function with $F(1) = 0$. Since we do not know \textit{a priori} $0\leq \phi \leq 1$, we extend $F$ to all of $\R$ by 0. This reduces the regularity of $F$, however, the only point of discontinuity is at 0, a fact we exploit later. As $F$ is a source term for the solute concentration function, $\phi$, in 2 dimensions, this can represent a `from the top' evaporation rate.

Our main result is the well-posedness of system \eqref{Eq:MainModel} and the fact that the qualitative property \eqref{Assump:MPhi} is conserved on $Q_T$. Particularly, we are interested in weak solutions to \eqref{Eq:MainModel}:
\begin{definition} \label{Def:MainSol}
    We say a pair $(m,\phi) \in \left( L^2(0,T;H^1_\sharp(\Omega)) \cap L^\infty(0,T;L^2(\Omega)) \right)^2$ with $(\p_t m, \p_t \phi) \in \left( L^2(0,T;H^{-1}_\sharp(\Omega)) \right)^2 $ is a weak solution to \eqref{Eq:MainModel} if $(m(t=0),\phi(t=0)) = (m_0,\phi_0)$ and for any $t \in (0,T)$, 
    \[   \langle \p_t m , \psi \rangle  + \int_\Omega \nabla m \cdot  \nabla \psi -  2\beta (\phi - m^2) (\nabla J* m) \cdot \nabla \psi \diff x  = 0, \] and
    \[  \langle \p_t \phi , \eta \rangle  + \int_\Omega  \nabla \phi \cdot  \nabla \eta -  2\beta m (1-\phi) (\nabla J* m) \cdot \nabla \eta \diff x  = \int_\Omega F(\phi) \eta \diff{x}, \]
    for all $(\psi,\eta) \in L^2(0,T;H^1_\sharp(\Omega)) \times L^2(0,T;H^1_\sharp(\Omega))$.
\end{definition}
\noindent Here and throughout the paper, we use the subscript $\sharp$ to restrict function spaces to the subset of spatially periodic functions. 
We aim to prove the following theorem:
\begin{theorem}\label{Thrm:MainThrm}
    Let $m_0 , \phi_0 $ be such that assumption \eqref{Assump:MPhi} holds. Then, there exists a unique weak solution $(m,\phi)$ to \eqref{Eq:MainModel} in the sense of Definition \ref{Def:MainSol}. Moreover, 
    \begin{equation}\label{Ineq:MPhi}
        0 \leq |m| \leq \phi \leq 1 \text{ almost everywhere in }Q_T.
    \end{equation}
\end{theorem}

To this end, assume for a moment that solutions to \eqref{Eq:MainModel} exist. Then the functions $w = \phi + m$ and $v = \phi - m$ would satisfy the system
\begin{equation}\label{Eq:SumDiff_System}
\left  \lbrace
\begin{split}
    \p_t w &= \Delta w + \text{div}[\beta (m-1) (\nabla J* (w-v)) w] + F(\tfrac{1}{2}(w+v)), \\
    \p_t v &= \Delta v + \text{div}[\beta(m+1) (\nabla J*(w-v)) v]+ F(\tfrac{1}{2}(w+v)), \\
    m  &= \frac{1}{2}(w-v),
\end{split} \right.   \quad (t,x) \in Q_T.
\end{equation}
This manipulation has revealed some hidden symmetry in our problem which turns out to be exploitable. Notice now that assumption \eqref{Assump:MPhi} implies certain bounds on the initial data $w_0,v_0$, specifically,
\begin{equation}\label{Assump:WV_data}
    0\leq w_0, v_0 \leq 2.
\end{equation}

It is easy to see that existence of solutions to system \eqref{Eq:SumDiff_System} implies the existence of solutions to the original problem \eqref{Eq:MainModel}. Moreover, the linear nature of the transformations between the systems allow us to carry over regularity properties on $(w,v)$ to $(m,\phi)$. Also, the conservation of inequality \eqref{Assump:WV_data} for all time will imply the conservation of part of \eqref{Assump:MPhi}, namely, $0 \leq |m| \leq \phi$. This combined with an \textit{a priori } result proven later will yield the conservation of \eqref{Assump:MPhi} for all $t\in [0,T)$.

Equation \eqref{Eq:SumDiff_System} is written in a peculiar way to foreshadow our plan of attack. We aim to use a Schauder fixed-point argument to prove the existence of solutions by finding a fixed-point on the composite map \[\mathcal{T}: \bar{m} \xmapsto{\mathcal{T}_1} (\bar{w},\bar{v}) \xmapsto{\mathcal{T}_2} \frac{1}{2}(\bar{w}-\bar{v}).\] Such arguments are common in literature, see e.g.,\cite{giacomin1998phase} for an application to a similar equation to the first in \eqref{Eq:MainModel} and \cite{eden2022multiscale} for an application to systems.

\subsection{Notation and organization of the paper}\label{Sec:Notation}

Positive constants whose explicit value is unimportant to the story will be denoted arbitrarily by $C$. If the constant depends on a particular function, $f$, this will be denoted by $C_f$. We warn that constants who share the same notation, need not share the same value. 

We liberally make use of Young's inequality for convolutions:
\begin{equation}\label{Eq:YoungsConvolution}
    \|f*g\|_r \leq \|f\|_p \|g\|_q,  \quad\text{where}\quad \frac{1}{p} +\frac{1}{q} = 1 + \frac{1}{r},
\end{equation}
whenever the above norms make sense for the functions in question.

For lucidity, we  use a short hand notation for spaces which appear frequently throughout the manuscript. For a given $\Omega \subset \mathbb{R^d}$, we define the following sets:
\[ \mathcal{X} : = L^2(0,T;H^1_\sharp(\Omega)) \cap L^\infty(0,T;L^2(\Omega));\]
\[\mathcal{W} : = \left \lbrace u \in L^2(0,T;H^1_\sharp(\Omega)) : \frac{\diff}{\diff{t}}u \in L^2(0,T;H^{-1}_\sharp(\Omega)) \right \rbrace .\]

The rest of paper is organized in the following way: in Section \ref{Sec:AuxProblem}, we show the well--posedness of an auxiliary problem via an iteration scheme; in Section \ref{Sec:FixedPoint} we prove the existence of a fixed--point to equation \eqref{Eq:SumDiff_System} and complete the proof of Theorem \ref{Thrm:MainThrm}; in Section \ref{Sec:Evap} we describe a special case of how one would model the evaporation process and provide some numerical simulations. Finally, in Section \ref{Sec:Discussion} we conclude with some future work and some discussion on the results.

\section{Analysis of auxiliary problems}\label{Sec:AuxProblem}
Before we can show the well--posedness of system \eqref{Eq:SumDiff_System} (and therefore system \eqref{Eq:MainModel}), we collect some useful auxiliary and \textit{a priori} results. First, while \textit{a priori} boundedness of solutions to \eqref{Eq:MainModel} is not immediately clear, we can infer that $\phi$ is bounded given that $m$ is bounded. This result will be useful to us later in showing the map $\mathcal{T}: \bar m \mapsto \frac{1}{2}(\bar w - \bar v)$ maps bounded functions to bounded functions. 
\begin{prop}\label{Prop:ConBound}
    Let $(m,\phi)$ be solutions to \eqref{Eq:MainModel} in the sense of Definition \ref{Def:MainSol}. If $|m| \leq 1$ almost everywhere in $Q_T$,  then it also holds $\phi \leq 1$ almost everywhere in $Q_T$.
\end{prop}
\begin{proof}
    Setting $\eta = [1-\phi]^- $ the weak form of $\phi$ from Definition \ref{Def:MainSol}, we see that due to the support of $F(\phi)$, $F(\phi)\eta = 0$. And so, taking the absolute value of the right-hand side of the weak form yields:
    \[ \frac{\diff}{\diff t} \| \eta \|_2^2 + \|\nabla \eta \|_2^2  \leq \int_\Omega |(2\beta m \eta \nabla J* m) \cdot \nabla \eta |\diff x .\]
    Using H\"older and Young's inequalities, we arrive at
    \[\frac{\diff}{\diff t} \| \eta \|_2^2 + \|\nabla \eta \|_2^2  \leq 2\beta^2 C_J  \|\eta\|_2^2 + \frac{1}{2}  \| \nabla \eta \|_2^2. \]
    A standard application of the Gr\"onwall inequality (see, e.g., \cite[Chapter 7.1]{Evans_pde}) and using the fact that $\eta (t=0) \equiv 0$ implies that $\eta = [1-\phi]^- = 0$ almost everywhere in $Q_T$.
\end{proof}

 With similar arguments and motivation, we show the following auxiliary result:
\begin{prop}\label{Prop:NonLocal_Bounded}
    Assume for some $B: (0,T) \times \Omega \longrightarrow \mathbb{R}$ with $\|B\|_{L^\infty(0,T; L^2(\Omega))} \leq C_B$ and $J\in C_+^2(\Omega)$, we have a weak solution $u \in L^2(0,T; H^1_\sharp (\Omega))$ to the equation
    \begin{equation*}
        \p_t u = \Delta u + \text{div}(B \nabla J* u), \quad (t,x) \in Q_T
    \end{equation*}
    equipped with periodic boundary conditions and where $u(t=0) = u_0 $ with $ |u_0| \leq C$ almost everywhere in $\bar \Omega$. Then, $|u| \leq C $ almost everywhere in $Q_T$.
\end{prop}
\begin{proof}
    From the weak form of the equation, we have for any $\eta \in L^2(0,T; H^1_\sharp (\Omega))$,
    \[ \langle \p_t u , \eta \rangle  + \int_\Omega \nabla u \cdot  \nabla \eta +  (B \nabla J* u) \cdot \nabla \eta \diff x  = 0 .\]
    Letting $\eta = [u - C]^+$, we have from the above,
    \[ \frac{d}{dt} \| \eta \|_2^2 + \|\nabla \eta \|_2^2  = -\int_\Omega (B \nabla J* u) \cdot \nabla \eta \diff x .\]
    By the assumptions on $J$, we have $\nabla J * C = 0$, we see that $\nabla J * u = \nabla J * (u-C)$. Therefore, using H\"older and Young's inequalities, 
    \begin{align*}
        \frac{d}{dt} \| \eta \|_2^2 + \|\nabla \eta \|_2^2 &\leq \|\nabla J * \eta \|_\infty \, \|B\|_2 \,\|\nabla \eta \|_2 \\
        &\leq \|\nabla J\|_2 \,\| \eta\|_2 \, \|B\|_2\, \|\nabla \eta \|_2 \\
        &\leq\frac{1}{2}C_J^2 \, C_B^2\, \| \eta\|_2^2 + \frac{1}{2}\|\nabla \eta \|_2^2 .
    \end{align*}
    A straightforward application of the Gr\"onwall's inequality implies that $\|\eta \|_2^2 = 0$ for all $t\in (0,T)$. Repeating this argument for the test function $\psi = [u + C ]^-$ implies the result.
    
\end{proof}

\subsection{Auxiliary system}
In this section, we prove the following system is well-posed:
\begin{equation}\label{Eq:AuxMain}
    \left  \lbrace
    \begin{split}
        \p_t w &= \Delta w + \text{div}[B_1 (\nabla J* (w-v)) w] + F(\tfrac{1}{2}(w+v)) \\
        \p_t v &= \Delta v + \text{div}[B_2 (\nabla J*(w-v)) v] + F(\tfrac{1}{2}(w+v))\\
    \end{split} \right.  , \quad (t,x) \in  Q_T,
\end{equation}
equipped with periodic boundary conditions, where for $i= 1,2$, $B_i : Q_T \longrightarrow \mathbb{R} \in L^\infty(Q_T)$ and  $w(0,x) = w_0(x)$, $v(0,x) = v_0(x)$ are bounded and nonnegative for all $x\in\bar\Omega$. Furthermore,  we set $\int_\Omega w_0 \diff x = K_w$, $\int_\Omega v_0 \diff x = K_v$ for some $K_w, K_v \geq 0$.  We understand solutions to equation \eqref{Eq:AuxMain} in the standard weak sense:
\begin{definition} \label{Def:AuxSol}
    We say a pair $(w,v) \in \mathcal{X} ^2$ with $(\p_t w, \p_t v) \in L^2(0,T;H^{-1}_\sharp(\Omega))^2 $ is a solution to \eqref{Eq:AuxMain} if $(w(t=0),v(t=0)) = (w_0,v_0)$ and for every $t \in (0,T)$, 
    \[   \langle \p_t w , \psi \rangle  + \int_\Omega \nabla w \cdot  \nabla \psi +  w B_1 (\nabla J* (w-v)) \cdot \nabla \psi \diff x  = \int_\Omega F(\tfrac{1}{2}(w+v)) \psi \diff{x}, \] and
    \[ \langle  \p_t v , \eta \rangle  + \int_\Omega \nabla v \cdot  \nabla \eta +  v B_2 (\nabla J* (w-v)) \cdot \nabla \eta \diff x  = \int_\Omega F(\tfrac{1}{2}(w+v)) \eta \diff{x}, \]
    for all $(\psi,\eta) \in L^2(0,T;H^1_\sharp(\Omega))^2$.
\end{definition}
Inspired by \cite{CHAZELLE2017365,Vera_2017}, we use an iteration scheme method to prove the following theorem
\begin{theorem}\label{Thrm:AuxExistence}
    There exist a unique non-negative weak solution in the sense of Definition \ref{Def:AuxSol} to \eqref{Eq:AuxMain} with the following estimates 
    \[ \| w\|_{L^\infty(0,T;L^2(\Omega))} +\|w\|_{L^2(0,T; H^1_\sharp(\Omega))} + \|\p_t w\|_{L^2(0,T; H^{-1}_\sharp(\Omega))} \leq C_{\Vec{B},J} (\|w_0\|_2 + \|v_0\|_2 + \|F\|_\infty \sqrt{T} ),\]
    and
    \[\| v\|_{L^\infty(0,T;L^2(\Omega))} +\|v\|_{L^2(0,T; H^1_\sharp(\Omega))} + \|\p_t v\|_{L^2(0,T; H^{-1}_\sharp(\Omega))} \leq C_{\Vec{B},J} (\|w_0\|_2 + \|v_0\|_2 + \|F\|_\infty \sqrt{T} ),\]
    where $\Vec{B}:= (B_1,B_2)$.
    Moreover, 
    \[\int_\Omega w \diff x \leq K_w + C_F t \text{ and } \int_\Omega v \diff x \leq  K_v + C_F t\text{ for all } t\in [0,T).\]
\end{theorem}

\begin{remark*}
    While equations in system \eqref{Eq:AuxMain} are quite similar to the equation studied by authors in \cite{CHAZELLE2017365}, they fix a particular function for their convolution and do not consider additional terms in the drift. This allows them certain estimates which are not applicable in our context.
\end{remark*}

\subsection{Iteration scheme}
We define the iteration scheme as follows:
\begin{equation}\label{Eq:Aux_IterationScheme}
    \left  \lbrace
    \begin{split}
        \p_t w_{n+1} &= \Delta w_{n+1} + \text{div}[B_1 w_{n+1} \, \nabla J* (w_{n}-v_{n}) ] + F(\tfrac{1}{2}(w_{n}+v_{n})), \\
        \p_t v_{n+1} &= \Delta v_{n+1} + \text{div}[B_2  v_{n+1} \, \nabla J*(w_{n}-v_{n})] + F(\tfrac{1}{2}(w_{n}+v_{n})), \\
    \end{split} \right.   \quad (t,x) \in  Q_T,
\end{equation}
where the initial conditions $w(0,x) = w_0(x)$ and $v(0,x) = v_0(x)$ are taken as the initial iteration step. We inherit all of the assumptions from equation \eqref{Eq:AuxMain} and define solutions to this system in a similar manner to Definition \ref{Def:AuxSol} with the corresponding weak form:
\begin{equation}\label{Eq:IterationScheme_WeakForm}
\left \lbrace
    \begin{split}
        \langle \p_t w_{n+1} , \psi \rangle  + \int_\Omega  \nabla w_{n+1} \cdot  \nabla \psi +  w_{n+1} B_1 (\nabla J* (w_{n}-v_{n})) \cdot \nabla \psi \diff x  = \int_\Omega F(\tfrac{1}{2}(w_{n}+v_{n})) \psi \diff{x},\\
        \langle\p_t v_{n+1} , \eta \rangle  +\int_\Omega \nabla v_{n+1} \cdot  \nabla \eta +  v_{n+1} B_2 (\nabla J* (w_{n}-v_{n})) \cdot \nabla \eta \diff x  = \int_\Omega F(\tfrac{1}{2}(w_{n}+v_{n})) \eta \diff{x}.
    \end{split} \right. 
\end{equation}

 We begin with an \textit{a priori} result:
\begin{lemma}
    Let $\{w_n\}$ and $\{v_n\}$ be sequences of solutions to \eqref{Eq:Aux_IterationScheme}. Then for each $n$, \[w_n, v_n \geq 0 , \int_\Omega w_n \diff x \leq  K_w + \|F\|_\infty |\Omega| t, \text{ and }\int_\Omega v_n \diff x \leq K_v + \|F\|_\infty |\Omega| t.\]
\end{lemma}
\begin{proof}
    We prove this statement using mathematical induction. By the assumptions on the data $w_0$ and $v_0$, the claim trivially holds. Now assume $w_n, v_n \geq 0$ with \[\int_\Omega w_n \diff x \leq  K_w + \|F\|_\infty |\Omega| t \text{  and } \int_\Omega v_n \diff x \leq  K_v + \|F\|_\infty |\Omega| t.\] Testing the equations in \eqref{Eq:IterationScheme_WeakForm} with $\psi = [w_{n+1} ]^-$ and $\eta = [v_{n+1} ]^-$, and summing them together, we have 
    \begin{align*}
        \frac{1}{2} \frac{\diff}{\diff t} \|\psi \|_2^2 + \| \nabla \psi \|_2^2 + \frac{1}{2} \frac{\diff}{\diff t} \|\eta \|_2^2 + \| \nabla \eta \|_2^2 &= -\int_\Omega B_1 \psi \nabla J * (w_n -v_n) \cdot \nabla \psi \diff x -\int_\Omega B_2 \eta \nabla J * (w_n -v_n) \cdot \nabla \eta \diff x\\
        & \quad + \int_\Omega F(\tfrac{1}{2}(w_{n}+v_{n})) (\psi + \eta) \diff{x}.
    \end{align*}
    Noticing due to the fact that $F$ is nonegative, we have 
    \[ \int_\Omega F(\tfrac{1}{2}(w_{n+1}+v_{n+1})) (\psi + \eta) \diff{x} \leq \int_\Omega F(\tfrac{1}{2}(w_{n}+v_{n})) [w_{n+1} + v_{n+1}]^- \diff{x} \leq 0.\]
    Therefore, we have the inequality
    \[ \frac{1}{2} \frac{\diff}{\diff t} \|\psi \|_2^2 + \| \nabla \psi \|_2^2 + \frac{1}{2} \frac{\diff}{\diff t} \|\eta \|_2^2 + \| \nabla \eta \|_2^2 \leq -\int_\Omega B_1 \psi \nabla J * (w_n -v_n) \cdot \nabla \psi \diff x -\int_\Omega B_2 \eta \nabla J * (w_n -v_n) \cdot \nabla \eta \diff x.    
    \]
    Making use of H\"older's inequality,  Young's inequality, and Young's inequality for convolution we have,
    \begin{align*}
        \frac{1}{2} \frac{\diff}{\diff t} \|\psi \|_2^2 + \frac{1}{2} \frac{\diff}{\diff t} \|\eta \|_2^2 &\leq \frac12 \| B_1 \nabla J * (w_n -v_n)\|_\infty^2  \|\psi\|_2^2  + \frac12 \| B_2 \nabla J * (w_n -v_n)\|_\infty^2  \|\eta\|_2^2  \\
        &\leq \frac12 (C_B C_J (K_w+K_v + 2\|F\|_\infty |\Omega|T))^2 (\|\psi\|_2^2 + \|\eta\|_2^2).
    \end{align*}
    Taking note that $\psi (0,x)\equiv 0 $ and $\eta (0,x)\equiv 0 $, the Gr\"onwall inequality implies that $\psi\equiv 0 $ and $\eta \equiv 0 $ in $Q_T$ and, therefore, $w_{n+1} , v_{n+1} \geq 0.$

    Testing equation \eqref{Eq:IterationScheme_WeakForm} with $\psi, \eta = 1$ yields 
    \[ \frac{\diff }{\diff t} \int_\Omega w_{n+1} \diff x = \int_\Omega F(\tfrac{1}{2}(w_{n}+v_{n}))\diff{x} \text{ and } \frac{\diff }{\diff t} \int_\Omega v_{n+1} \diff x = \int_\Omega F(\tfrac{1}{2}(w_{n}+v_{n}))\diff{x}.\] Since $w_{n+1}(0,x) = w_0(x)$ and $v_{n+1}(0,x) = v_0(x)$ we have \[\int_\Omega w_{n+1} \diff x = \int_\Omega w_{0} \diff x + \int_0^t\int_\Omega F(\tfrac{1}{2}(w_n(s,\cdot)+v_n(s,\cdot)))\diff{x} \diff{t} \leq K_w + \|F\|_\infty |\Omega| t\] and \[\int_\Omega v_{n+1} \diff x = \int_\Omega v_{0} \diff x + \int_0^t\int_\Omega F(\tfrac{1}{2}(w_n(s,\cdot)+v_n(s,\cdot)))\diff{x} \diff{t} \leq K_v + \|F\|_\infty |\Omega| t\] for all $t \in (0,T)$.    
\end{proof}

Due to Young's convolution inequality \eqref{Eq:YoungsConvolution}, we have that the drift term in \eqref{Eq:Aux_IterationScheme} is in fact uniformly bounded. Indeed, \[\| B_i \nabla J * ( w_n -v_n)\|_{L^\infty(Q_T)} \leq \|B_i\|_{L^\infty(Q_T)} \|\nabla J\|_\infty (K_{w} + K_{v} +2|\Omega|\,  T\, \|F\|_\infty).\] Therefore, inductively applying standard Galerkin arguments found in \cite[Section 7.1]{Evans_pde}, we have for each $n$ a unique solution pair $(w_{n+1}, v_{n+1}) \in \mathcal{X}^2$ in the sense of \eqref{Eq:IterationScheme_WeakForm} and with energy estimates similar to those found in Theorem \ref{Thrm:AuxExistence}. Specifically, we have the energy estimates:
    \[ \| w_n\|_{L^\infty(0,T;L^2(\Omega))} +\|w_n\|_{L^2(0,T; H^1_\sharp(\Omega))} + \|\p_t w_n\|_{L^2(0,T; H^{-1}_\sharp(\Omega))} \leq C_{\Vec{B},J} (\|w_0\|_2 + \|v_0\|_2 + \|F\|_{\infty} \sqrt{T} ),\]
    and
    \[\| v_n\|_{L^\infty(0,T;L^2(\Omega))} +\|v_n\|_{L^2(0,T; H^1_\sharp(\Omega))} + \|\p_t v_n\|_{L^2(0,T; H^{-1}_\sharp(\Omega))} \leq C_{\Vec{B},J} (\|w_0\|_2 + \|v_0\|_2 +\|F\|_{\infty} \sqrt{T}),\]
    for each $n = 1,2,\dots .$

We now ensure that our scheme is convergent by proving that it is a Cauchy sequence in a suitable space of functions. Notice that since we have shown $w_n, v_n \geq 0$, we may use the Lipschitz continuity of $F$.

\begin{lemma}\label{Lem:Cauchy}
    There exists a $t^*>0$ such that the sequence $(w_n, v_n)$ is Cauchy in $L^\infty(0,t^*;L^2(\Omega))$.
\end{lemma}
\begin{proof}
    Let $\Bar{w}_n : = w_{n} - w_{n-1}$ and $\Bar{v}_n := v_n - v_{n-1}$. Then from \eqref{Eq:IterationScheme_WeakForm} we have that $\Bar{w}_{n+1}$ and $\Bar{v}_{n+1}$  satisfy the weak form:
\begin{equation}\label{Eq:ConsecDiff_WeakForm}
\left \lbrace
    \begin{split}
        \langle  \p_t \Bar{w}_{n+1} , \psi \rangle   +\int_\Omega \nabla \Bar{w}_{n+1} \cdot  \nabla \psi + w_{n} B_1 (\nabla J* (\Bar{w}_{n}-\Bar{v}_{n})) \cdot \nabla \psi +  \Bar{w}_{n+1} B_1 (\nabla J* (w_{n}-v_{n})) \cdot \nabla \psi \diff x \\  = \int_\Omega [F(\tfrac{1}{2}(w_n + v_n)) - F(\tfrac{1}{2}(w_{n-1} + v_{n-1}))]\psi \diff{x},\\
     \langle \p_t \Bar{v}_{n+1} , \eta \rangle  + \int_\Omega \nabla \Bar{v}_{n+1} \cdot  \nabla \eta + v_{n} B_2 (\nabla J* (\Bar{w}_{n}-\Bar{v}_{n}))\cdot \nabla \eta +  \Bar{v}_{n+1} B_2 (\nabla J* (w_{n}-v_{n})) \cdot \nabla \eta \diff x  \\ \quad = \int_\Omega [F(\tfrac{1}{2}(w_n + v_n)) - F(\tfrac{1}{2}(w_{n-1} + v_{n-1}))]\eta \diff{x}.
    \end{split} \right. 
\end{equation}
The above formulation allows us by testing with $(\psi,\eta) = (\bar w_{n+1},\bar v_{n+1})$ to arrive at the identities 
\begin{equation}\label{Eq:CauchyProof1}
    \left \lbrace
    \begin{split}
        \frac{\diff}{\diff t}\|\Bar{w}_{n+1} \|_2^2 + \| \nabla \Bar{w}_{n+1} \|_2^2 &= - \int_\Omega w_{n} B_1 (\nabla J* (\Bar{w}_{n}-\Bar{v}_{n})) \cdot \nabla \Bar{w}_{n+1}  +  \Bar{w}_{n+1} B_1 (\nabla J* (w_{n}-v_{n})) \cdot \nabla \Bar{w}_{n+1} \diff x \\
        &\qquad + \int_\Omega [F(\tfrac{1}{2}(w_n + v_n)) - F(\tfrac{1}{2}(w_{n-1} + v_{n-1}))] \Bar{w}_{n+1} \diff{x}, \\
        \frac{\diff}{\diff t}\|\Bar{v}_{n+1} \|_2^2 + \| \nabla \Bar{v}_{n+1}\|_2^2 &= - \int_ \Omega v_{n} B_2 (\nabla J* (\Bar{w}_{n}-\Bar{v}_{n}))\cdot \nabla \Bar{v}_{n+1}  +  \Bar{v}_{n+1} B_2 (\nabla J* (w_{n}-v_{n})) \cdot \nabla \Bar{v}_{n+1} \diff x \\
        &\qquad + \int_\Omega [F(\tfrac{1}{2}(w_n + v_n)) - F(\tfrac{1}{2}(w_{n-1} + v_{n-1}))] \Bar{v}_{n+1} \diff{x}.
    \end{split} \right. 
\end{equation}
Working with the right-hand side of the first equation above we have the first term,
\begin{align*}
    \int_\Omega w_{n} B_1 (\nabla J* (\Bar{w}_{n}-\Bar{v}_{n})) \cdot \nabla \Bar{w}_{n+1} \diff x &\leq \| B_1 (\nabla J* (\Bar{w}_{n}-\Bar{v}_{n})) \|_\infty \int_\Omega |w_{n} \nabla \Bar{w}_{n+1} | \diff x  \\
    &\leq C_{\Vec{B}} C_J (\|\Bar{w}_n \|_2 + \|\Bar{v}_n \|_2) \, \| w_n \|_2 \, \|\nabla \Bar{w}_{n+1} \|_2,
\end{align*}
where in the last inequality, we have used H\"older and Young's inequality for convolution. 
Making use of the weighted Young's inequality and the energy estimate on $w_n$, we can arrive at
\begin{align*}
    \int_\Omega w_{n} B_1 (\nabla J* (\Bar{w}_{n}-\Bar{v}_{n})) \cdot \nabla \Bar{w}_{n+1} \diff x &\leq C_{\Vec{B},J,F,w_0,v_0} (\|\Bar{w}_n \|_2^2 + \|\Bar{v}_n \|_2^2) + \frac{1}{2} \|\nabla \Bar{w}_{n+1} \|_2^2 .
\end{align*}

Looking now to the second term, we have via Young's inequality for convolutions, H\"older's inequality, and Young's product inequality
\begin{align*}
    \int_\Omega \Bar{w}_{n+1} B_1 (\nabla J* (w_{n}-v_{n})) \cdot \nabla \Bar{w}_{n+1} \diff x &\leq \|\nabla J* (w_{n}-v_{n})\|_\infty \int_\Omega |\Bar{w}_{n+1} \nabla \Bar{w}_{n+1}| \diff x \\
    &\leq C_{J,F,w_0,v_0}\| \Bar{w}_{n+1} \|_2^2  + \frac{1}{2} \|\nabla \Bar{w}_{n+1} \|_2^2. 
\end{align*}

Finally, from the third term on the right-hand side we see by Lipschitz continuity of $F$ (denoting the Lipschitz constant by $L_F$), 
\begin{align*}
     \int_\Omega [F(\tfrac{1}{2}(w_n + v_n)) - F(\tfrac{1}{2}(w_{n-1} + v_{n-1}))] \Bar{w}_{n+1} \diff{x} &\leq L_F \int_\Omega \tfrac{1}{2}|\Bar{w}_n + \Bar{v}_n| |\bar{w}_n| \diff{x}\\
     &\leq C_F (\|\Bar{w}_n\|_2^2 + \|\Bar{v}_n\|_2^2 + \|\Bar{w}_{n+1}\|_2^2),   
\end{align*}
where in the last step we again make use of H\"older's inequality and Young's inequality.

Repeating these calculations for $\Bar{v}_{n+1}$ and summing the equations in \eqref{Eq:CauchyProof1}, we arrive at the estimates
\begin{equation}\label{Eq:DiffEstimate}
    \frac{\diff}{\diff t} \| \Bar{w}_{n+1} \|_2^2 + \frac{\diff}{\diff t} \| \Bar{v}_{n+1} \|_2^2 \leq C_{\Vec{B},J,F,w_0,v_0} (\| \Bar{w}_{n} \|_2^2 + \| \Bar{v}_{n} \|_2^2+ \| \Bar{w}_{n+1} \|_2^2 +\| \Bar{v}_{n+1} \|_2^2 ).
\end{equation}
Choosing $t^*$ such that 
\begin{equation*}
    \frac{C_{\Vec{B},J,F,w_0,v_0}t^*}{1-C_{\Vec{B},J,F,w_0,v_0}t^*} < 1
\end{equation*}
and integrating \eqref{Eq:DiffEstimate} from 0 to $t <t^*$ we can arrive at the estimate
\begin{equation*}
    \|\Bar{w}_{n+1} (t,\cdot) \|_2^2 + \|\Bar{v}_{n+1} (t,\cdot) \|_2^2 \leq C_{\Vec{B},J,F,w_0,v_0} \left( \int_0^t \|\Bar{w}_{n} (s,\cdot) \|_2^2 + \|\Bar{v}_{n} (s,\cdot) \|_2^2 \diff s+\int_0^t \|\Bar{w}_{n+1} (s,\cdot) \|_2^2  + \|\Bar{v}_{n+1} (s,\cdot) \|_2^2 \diff s \right).
\end{equation*}
Taking the supremum of both sides over $[0,t^*)$, we obtain:
\begin{align*}
    \|\Bar{w}_{n+1}  \|^2_{L^\infty(0,t^*;L^2(\Omega))} + \|\Bar{v}_{n+1} 
    \|_{L^\infty(0,t^*;L^2(\Omega))}^2 &\leq C_{\Vec{B},J,F,w_0,v_0} \left(  \|\Bar{w}_{n}   \|_{L^2(0,t^*;L^2(\Omega))}^2 + \|\Bar{v}_{n} \|_{L^2(0,t^*;L^2(\Omega))}^2 \right. \\
    &\quad+\left. \|\Bar{w}_{n+1} \|_{L^2(0,t^*;L^2(\Omega))}^2  + \|\Bar{v}_{n+1}  \|_{L^2(0,t^*;L^2(\Omega))}^2 \right) \\
    &\leq C_{\Vec{B},J,F,w_0,v_0} t^* \left(  \|\Bar{w}_{n}   \|_{L^\infty(0,t^*;L^2(\Omega))}^2 + \|\Bar{v}_{n} \|_{L^\infty(0,t^*;L^2(\Omega))}^2 \right. \\
    &\quad+\left. \|\Bar{w}_{n+1} \|_{L^\infty(0,t^*;L^2(\Omega))}^2  + \|\Bar{v}_{n+1}  \|_{L^\infty(0,t^*;L^2(\Omega))}^2 \right),
\end{align*}
where we have used the inequality 
\begin{equation}\label{Eq:Linfty2_inequality}
    \| u\|_{L^2(0,t^*;L^2(\Omega))}^2 \leq t^* \| u\|_{L^\infty(0,t^*;L^2(\Omega))}^2.
\end{equation}
Manipulation of the above inequality yields
\begin{equation}
    \|\Bar{w}_{n+1}  \|^2_{L^\infty(0,t^*;L^2(\Omega))} + \|\Bar{v}_{n+1} 
    \|_{L^\infty(0,t^*;L^2(\Omega))}^2 \leq \frac{C_{\Vec{B},J,F,w_0,v_0} t^*}{1-C_{\Vec{B},J,F,w_0,v_0} t^*} \left(  \|\Bar{w}_{n}   \|_{L^\infty(0,t^*;L^2(\Omega))}^2 + \|\Bar{v}_{n} \|_{L^\infty(0,t^*;L^2(\Omega))}^2 \right),
\end{equation}
which implies the sequence is Cauchy in $L^\infty(0,t^*;L^2(\Omega))$. Moreover, using inequality \eqref{Eq:Linfty2_inequality}, we see this implies the sequence is also Cauchy in $L^2(0,t^*;L^2(\Omega))$.
\end{proof}

\begin{remark*}
    The uniqueness of solutions to \eqref{Eq:AuxMain} follows from similar arguments presented in the above proof. Indeed, letting $\Bar{w} = w_1 - w_2$ and $\Bar{v} = v_1 - v_2$ where $(w_1,v_1)$ and $(w_2,v_2)$ are solutions to \eqref{Eq:AuxMain}, we see that $\Bar{w}$ and $\Bar{v}$ satisfy similar equations to \eqref{Eq:ConsecDiff_WeakForm}, namely
    \begin{equation*}
\left \lbrace
    \begin{split}
             \langle  \p_t \Bar{w} , \psi \rangle  + \int_\Omega \nabla \Bar{w} \cdot  \nabla \psi + w_{2} B_1 (\nabla J* (\Bar{w}-\Bar{v})) \cdot \nabla \psi+  \Bar{w} B_1 (\nabla J* (w_{1}-v_{1})) \cdot \nabla \psi \diff x \\
              = \int_\Omega [F(\tfrac{1}{2}(w_1 + v_1)) - F(\tfrac{1}{2}(w_{2} + v_{2}))]\psi \diff{x},\\
     \langle \p_t \Bar{v} , \eta \rangle  + \int_\Omega \nabla \Bar{v} \cdot  \nabla \eta + v_{2} B_2 (\nabla J* (\Bar{w}-\Bar{v}))\cdot \nabla \eta +  \Bar{v} B_2 (\nabla J* (w_{1}-v_{1})) \cdot \nabla \eta \diff x  \\
              = \int_\Omega [F(\tfrac{1}{2}(w_1 + v_1)) - F(\tfrac{1}{2}(w_{2} + v_{2}))]\eta \diff{x}.
    \end{split} \right. 
\end{equation*}Following the calculations which lead to equation \eqref{Eq:DiffEstimate}, we arrive at the differential inequality
    \[ \frac{\diff}{\diff t} \| \Bar{w} \|_2^2 + \frac{\diff}{\diff t} \| \Bar{v} \|_2^2 \leq C_{\Vec{B},J,F,w_0,v_0} (\| \Bar{w} \|_2^2 + \| \Bar{v} \|_2^2).\]
    A standard application of Gr\"onwall's inequality implies $\| \Bar{w} \|_2^2 +  \| \Bar{v} \|_2^2 = 0$ for all $t \in [0,T]$.
\end{remark*}

\begin{remark*}
    By continuity and the above energy estimates, we can extend the solution to any finite $T$ using standard methods (see, e.g., \cite[Section 7]{amann2011ordinary}).
\end{remark*}

Summing up, we have obtained that the constructed sequence $(w_n,v_n)$ is in fact a  Cauchy sequence in $L^\infty(0,T;L^2(\Omega))^2 \text{ (or }L^2(Q_T)^2)$ and, from the energy estimates, is weakly convergent (along a subsequence) in $L^2(0,T;H^1_\sharp(\Omega))$ with $(\p_t w_n, \p_t v_n)$ weakly convergent (along a subesquence) in $L^2(0,T;H^{-1}_\sharp(\Omega))$. Using these results, we can pass to the limit in the weak formulation \eqref{Eq:IterationScheme_WeakForm} (following the calculations used to derive equation \eqref{Eq:DiffEstimate}) as $n \longrightarrow \infty $ to arrive at the weak formulation of Definition \ref{Def:AuxSol}. Thus, the limit is a solution to \eqref{Eq:AuxMain}.

\section{The fixed-point argument}\label{Sec:FixedPoint}
We are now ready to prove the well-posedness of equation \eqref{Eq:SumDiff_System} which for the reader's convenience we rewrite here:
\begin{equation}
\left  \lbrace
\begin{split}
    \p_t w &= \Delta w + \text{div}[\beta (m-1) (\nabla J* (w-v)) w] + F(\tfrac{1}{2}(w+v)), \\
    \p_t v &= \Delta v + \text{div}[\beta(m+1) (\nabla J*(w-v)) v]+ F(\tfrac{1}{2}(w+v)), \\ 
    m  &= \frac{1}{2}(w-v),
\end{split} \right.   \quad (t,x) \in Q_T. \tag{\ref{Eq:SumDiff_System}}
\end{equation}
Taking into account  assumption \eqref{Assump:WV_data}, solutions to \eqref{Eq:SumDiff_System} are understood as in Definition \ref{Def:AuxSol} with $B_1 := \beta (m-1)$ and $B_2 := \beta (m+1)$. We aim to show the existence of solutions to  a fixed-point to the map $\mathcal{T}: \mathcal{Z} \longrightarrow \mathcal{Z}$ given by 
\[\bar{m} \xmapsto{\mathcal{T}_1} (\bar{w},\bar{v}) \xmapsto{\mathcal{T}_2} \frac{1}{2}(\bar{w}-\bar{v}),\]
where 
\[\mathcal{Z} := \{ u \in L^2(Q_T) : -1 \leq u \leq 1 \text{ almost everywhere in } Q_T\} ,\]

$\mathcal{T}_1(\bar m) := (\bar w, \bar v)$ is the weak solution to \eqref{Eq:AuxMain} with $B_1 = \beta (\bar m -1)$ and $B_2 = \beta (\bar m + 1)$, and $\mathcal{T}_2(\bar w, \bar v) = \frac{1}{2}(\bar w - \bar v)$.

Now, using the work done in the previous sections, we have the following lemmas:

\begin{lemma}\label{Lem:T_Well-defined}
    The map $\mathcal{T} = \mathcal{T}_2 \circ \mathcal{T}_1 $ is well-defined over the set $\mathcal{Z}$ and $\mathcal{T}:\mathcal{Z} \longrightarrow \mathcal{Z}$.  
\end{lemma}
\begin{proof}
We study separately the maps $\mathcal{T}_1$ and $\mathcal{T}_2$.  

\subsubsection*{The map $\mathcal{T}_1$ :} Fix $\bar m \in \mathcal{Z}$. By Theorem \ref{Thrm:AuxExistence}, we have the existence of a unique pair 
$$(\bar w , \bar v) \in \mathcal{X}^2 \text{ with } (\p_t \bar w , \p_t \bar v) \in \left( L^2(0,T;H^{-1}_\sharp(\Omega)) \right)^2$$ 
such that $0 \leq \bar w , \bar v$. Thus the map $\mathcal{T}_1$ is well-defined.

\subsubsection*{The map $\mathcal{T}_2$ :} The map $\mathcal{T}_2$ is trivially well-defined. What remains to be shown is that $\mathcal{T}_2$ maps into $\mathcal{Z}$. To see this, define $m := \frac{1}{2}(\bar w-\bar v) \in L^2(Q_T) $ and notice that
\begin{equation}
    \p_t m = \Delta m - \text{div} [\beta (\bar w+\bar v -\bar m (\bar w-\bar v)) \nabla J*m],
\end{equation}
with $|m_0| \leq 1.$ Since $\bar m \in L^\infty(Q_T)$ and $\bar w , \bar v \in  L^2(0,T;H^1_\sharp(\Omega)) \cap L^\infty(0,T;L^2(\Omega))$, we have $\beta (\bar w+\bar v -\bar m (\bar w-\bar v)) =: B(t,x) \in L^\infty(0,T;L^2(\Omega))$ . Thus, by Proposition \ref{Prop:NonLocal_Bounded}, we have that $|m| \leq 1$.     
\end{proof}

\begin{remark}\label{Rem:wv_Bounded}
    With similar arguments to those found in Proposition \ref{Prop:ConBound}, it can be easily shown the function $\phi := \frac{1}{2}(\bar w + \bar v)$, which satisfies the equation\[ \p_t \phi = \Delta \phi - \text{div} [2\beta \Bar{m} (1-\phi) \nabla J*m] + F(\phi), \]
    is nonnegative and bounded by 1 (i.e.,$0\leq \phi \leq 1$) since $\Bar{m}, m \in \mathcal{Z}$ and $0\leq \bar w, \bar v$.
    Together, these bounds allow us to conclude $0\leq \bar w, \bar v \leq 2$ almost everywhere in $Q_T$.
\end{remark}

\begin{lemma}\label{Lem:T_Continuous}
    The map $\mathcal{T} = \mathcal{T}_2 \circ \mathcal{T}_1 : \mathcal{Z} \longrightarrow \mathcal{Z}$ is continuous.
\end{lemma}
\begin{proof}
    Let $\bar m_n , \bar m \in \mathcal{Z}$ be such that $ \bar m_n \longrightarrow \bar  m$ in $L^2(Q_T)$. Define also $(\bar w_n, \bar v_n) := \mathcal{T}_1 (m_n)$ and $(\bar w, \bar v) := \mathcal{T}_1(m).$ We want to show $(\bar w_n, \bar v_n) \longrightarrow ( \bar w,\bar v)$ in $L^2(Q_T).$ To this end, let $m : =  \bar m_n - \bar m$, $w : = \bar w_n -\bar w $, and $v : = \bar v_n -\bar v $ for fixed $n$. Following similar manipulations which gave rise to equation \eqref{Eq:ConsecDiff_WeakForm}, we have that $w$ and $v$ satisfy
    \begin{equation}
    \left \lbrace
    \begin{split}
             \langle  \p_t w , \psi \rangle   + \int_\Omega \nabla w \cdot  \nabla \psi \diff x &= - \int_\Omega \beta \bar w_n (\bar m_n -1) (\nabla J* (w-v)) \cdot \nabla \psi+  w (\bar m_n -1) (\nabla J* (\Bar{w}-\Bar{v}))\cdot \nabla \psi \diff x \\
             &\qquad- \int_\Omega \bar w\, m (\nabla J* (\Bar{w}-\Bar{v}))  \cdot \nabla \psi \diff x  +\int_\Omega [F(\tfrac{1}{2}(\Bar{w}_n+\Bar{v}_n) - F(\tfrac{1}{2}(\Bar{w}+\Bar{v}) )]\psi \diff{x}, \\
     \langle  \p_t v , \eta \rangle  + \int_\Omega \nabla v \cdot  \nabla \eta \diff x  &= -\int_\Omega \beta  \bar v_n (\bar m_n +1) (\nabla J* (w-v))\cdot \nabla \eta  +  v (\bar m_n +1) (\nabla J* (\Bar{w}-\Bar{v})) \cdot \nabla \eta  \diff x\\
     &\qquad - \int_\Omega \bar v\, m (\nabla J* (\Bar{w}-\Bar{v})) \cdot \nabla \eta \diff x +\int_\Omega [F(\tfrac{1}{2}(\Bar{w}_n+\Bar{v}_n) - F(\tfrac{1}{2}(\Bar{w}+\Bar{v}))] \eta \diff{x}.
    \end{split} \right. 
    \end{equation}
    Testing with $(\psi, \eta) = (w,v)$ we have the following equations:
    \begin{equation}
        \left \lbrace
        \begin{split}
            \frac{\diff}{\diff t}\|w \|_2^2 + \| \nabla w \|_2^2 &= - \int_\Omega \beta  \bar w_n (\bar m_n -1) (\nabla J* (w-v)) \cdot \nabla w +w (\bar m_n -1) (\nabla J* (\Bar{w}-\Bar{v}))\cdot \nabla w \diff x \\ 
            &\qquad -\int_\Omega    \bar w\, m (\nabla J* (\Bar{w}-\Bar{v}))  \cdot \nabla w \diff x +\int_\Omega [F(\tfrac{1}{2}(\Bar{w}_n+\Bar{v}_n) - F(\tfrac{1}{2}(\Bar{w}+\Bar{v}) )]w \diff{x},\\      
            \frac{\diff}{\diff t}\|v \|_2^2 + \| \nabla v\|_2^2 &= - \int_ \Omega \beta  \bar v_n (\bar m_n +1) (\nabla J* (w-v))\cdot \nabla v +  v (\bar m_n +1) (\nabla J* (\Bar{w}-\Bar{v})) \cdot \nabla v \diff x  \\
            &\qquad - \int_\Omega \bar v\, m (\nabla J* (\Bar{w}-\Bar{v})) \cdot \nabla v \diff x +\int_\Omega [F(\tfrac{1}{2}(\Bar{w}_n+\Bar{v}_n) - F(\tfrac{1}{2}(\Bar{w}+\Bar{v}) )]v \diff{x}.
        \end{split} \right. 
    \end{equation}
    Making use of Remark \ref{Rem:wv_Bounded}, the first two terms on the right-hand side of both equations can be handled as in the proof of Lemma \ref{Lem:Cauchy} to arrive at the bounds:
   \begin{equation}
        \left \lbrace
        \begin{split}
             \left| \int_\Omega \beta \left[  \bar w_n (\bar m_n -1) (\nabla J* (w-v)) +  w (\bar m_n -1) (\nabla J* (\Bar{w}-\Bar{v}))  \right] \cdot \nabla w \diff x  \right|& \leq C_{J,F} (\|w\|_2^2 + \|v\|_2^2) + \frac{1}{3} \|\nabla w\|_2^2,\\
             \left| \int_ \Omega \beta \left[  \bar v_n (\bar m_n +1) (\nabla J* (w-v)) +  v (\bar m_n +1) (\nabla J* (\Bar{w}-\Bar{v}))  \right] \cdot \nabla v\diff x  \right| &\leq C_{J,F} (\|w\|_2^2 + \|v\|_2^2) + \frac{1}{3} \|\nabla v\|_2^2.
        \end{split} \right.
    \end{equation} 
    Similarly, the evaporation terms follow the same arguments,
   \begin{equation}
        \left \lbrace
        \begin{split}
             \left| \int_\Omega [F(\tfrac{1}{2}(\Bar{w}_n+\Bar{v}_n) - F(\tfrac{1}{2}(\Bar{w}+\Bar{v}) )]w\diff x  \right|& \leq C_F (\|w\|_2^2 + \|v\|_2^2 ),\\
             \left| \int_\Omega [F(\tfrac{1}{2}(\Bar{w}_n+\Bar{v}_n) - F(\tfrac{1}{2}(\Bar{w}+\Bar{v}) )]v\diff x  \right|& \leq C_F (\|w\|_2^2 + \|v\|_2^2 ).
        \end{split} \right. 
    \end{equation} 

    Meanwhile, the new third term yields the bound

        \begin{equation}
        \left \lbrace
        \begin{split}
             \left| \int_\Omega \bar w\, m (\nabla J* (\Bar{w}-\Bar{v})) \cdot \nabla w  \diff x  \right|& \leq C_{J,F} \|m\|_2^2 + \frac{1}{3} \|\nabla w\|_2^2,\\
             \left| \int_ \Omega \beta \bar v\, m (\nabla J* (\Bar{w}-\Bar{v}))\cdot \nabla v\diff x  \right| &\leq C_{J,F} \|m\|_2^2  + \frac{1}{3} \|\nabla v\|_2^2.
        \end{split} \right. 
    \end{equation}
    Summing up the two inequalities and maximizing constants leaves us with 
    \begin{equation}\label{Ineq:Lemma32proof}
        \frac{\diff}{\diff t}\|w \|_2^2 +\frac{\diff}{\diff t}\|v \|_2^2 \leq C_{J,F}( \|m\|_2^2 +\|w\|_2^2 +\|v\|_2^2).
    \end{equation}
    Integrating the above equation from $0$ to $t$ gives us the integral inequality 
    \begin{equation}
        \|w \|_2^2 +\|v \|_2^2 \leq C_{J,F} \|m\|_{L^2((0,t)\times \Omega)}^2 +C_{J,F}\int_0^t \|w\|_2^2 +\|v\|_2^2\diff s.
    \end{equation}
    Gr\"onwall's inequality and integration provides us with the estimate:
    \begin{equation}
        \|w \|_{L^2(Q_T)}^2 +\|v \|_{L^2(Q_T)}^2 \leq C_{J,F} T \|m\|_{L^2(Q_T)}^2 \exp{(C_{J,F} T)}.
    \end{equation}
    Continuity of $\mathcal{T}_1$ and therefore $\mathcal{T}$ follows.
\end{proof}

\begin{theorem}
    The map $\mathcal{T} := \mathcal{T}_2 \circ \mathcal{T}_1$ has at least one fixed-point  $m^* \in \mathcal{W} \cap \mathcal{Z}.$
\end{theorem}
\begin{proof}
    Summing up the results from before, we have the set $\mathcal{Z}$ is a non-empty, closed, and convex subset of $L^2(Q_T)$, the map $\mathcal{T}: \mathcal{Z} \longrightarrow \mathcal{Z}$ is well-defined and continuous (Lemmas \ref{Lem:T_Well-defined} and  \ref{Lem:T_Continuous}), and $\mathcal{T}(\mathcal{Z}) \subseteq \mathcal{W}$ which is compactly embedded in $L^2(Q_T)$ via Lions-Aubin's Lemma. Therefore, we can employ Schauder's Fixed-Point Theorem to conclude there exists a fixed-point in $\mathcal{W} \cap \mathcal{Z}.$
\end{proof}

Summing up, we have shown the existence of bounded solutions $(w,v)$ to problem \eqref{Eq:SumDiff_System}. By manipulation, we can see that the functions $m = \frac{1}{2}(w-v)$ and $\phi = \frac{1}{2}(w+v)$ are solutions to equation \eqref{Eq:MainModel}. Moreover, we have shown that $0 \leq w,v \leq 2$ which implies that 
\[ |m| \leq \phi .\]
Since in the fixed-point argument we show that $m \in \mathcal{Z}$, we can conclude via Proposition \ref{Prop:ConBound} that $\phi \leq 1$. Thus we have shown the existence of bounded solutions to equation \eqref{Eq:MainModel} in the sense of Definition \ref{Def:MainSol}. A proof of the uniqueness of bounded solutions to system \eqref{Eq:MainModel_withoutEvap} can be found in \cite[Lemma 5.4]{Marra}. However, one can also prove uniqueness in $L^2$ for system \eqref{Eq:MainModel} with similar arguments to Lemma \ref{Lem:T_Continuous} above. Indeed, we have
\begin{lemma}
    The solution pair $(w,v)$ given by Definition \ref{Def:AuxSol} is unique.
\end{lemma}
\begin{proof}
    Letting $w = w_1 - w_2$ and $v = v_1 - v_2$ be the difference of two solutions of \eqref{Eq:MainModel}, following the calculations in Lemma \ref{Lem:T_Continuous}, and using that $m = \tfrac{1}{2}(w-v)$, we have that inequality \eqref{Ineq:Lemma32proof} now reads
\[\frac{\diff}{\diff t}\|w \|_2^2 +\frac{\diff}{\diff t}\|v \|_2^2 \leq C_{J,F}( \|w\|_2^2 +\|v\|_2^2).\]
Gr\"onwall's inequality then proves the result.
\end{proof}

This proves Theorem \ref{Thrm:MainThrm}.

\section{A simple example of `from the top' evaporation}\label{Sec:Evap}

Looking at a thin film, modeling the evaporation of the solvent phase from a mixture of multiple phases interacting among themselves as well as with the solvent, is, in general, quite complicated. Consequently, accurate quantitative predictions are usually out of reach. However, in an experimentally controlled setup with not too many interacting phases and where the effect of temperature gradients can be neglected, valuable qualitative insights can be reached. For instance, in the context of applications to organic solar cells, the solvent evaporates into vapor and moves up and leaves the film. This process leaves us with two perspectives on the actual geometry of the thin film: from `the side' and from `the top'. For continuum models, the side perspective of the evaporation process can be modeled as a moving interface separating the solvent vapor and liquid phases as done, for instance, in \cite{cummings2018modeling}. This process tends to result in the formation of solvent `lanes' and can be observed both in continuum \cite{cummings2018modeling} models as well as in discrete lattice models \cite{Mario}. On the other hand, an observer of the top view of the film would see a bulk evaporation which is usually modeled as right-hand side source term of a balance law for solvent posed in two dimensions and not as a moving-boundary condition imposed an a moving surface within the thin film. Such situation is conceptually simpler and is easier to handle from modeling, mathematical analysis, and simulation points of view. In this section, we use a simple finite volume method (previously tested on model \eqref{Eq:MainModel_withoutEvap} in \cite{LyonsMunteanetal2023}) to demonstrate `from the top' evaporation and the effect it has on the morphologies. We leave the `from the side' and the 3 dimensional combination of the two perspectives for future work.

\subsection{Simulation results}\label{simulation}

To fix ideas, we take into account a linear evaporation model, i.e., $F:\mathbb{R}\to \mathbb{R}$ defined by $F(r) = \alpha(1-r)$ for all $r\in\mathbb{R}$ with given $\alpha>0$. We refer the reader to section 2.4.3.3 in \cite{AdrianContinuum} for a motivation of this particular structure that describes in a very simplified way the liquid-gas transition. More information on how to include the evaporation mechanism in models capturing phase separation in thin films can be found e.g. in \cite{Schaefer}. 

We use a finite volume scheme similar to that presented in \cite{LyonsMunteanetal2023} with slight modifications to include the production term by evaporation. Namely, for fixed mesh sizes $\Delta t, \Delta x, \Delta y > 0$ we discretize the  domain $\Omega$ by cells $\Lambda_{i,j} := [x_i -\frac12 h, x_i +\frac12 h) \times [y_j -\frac12 h, y_j +\frac12 h) $ for $i,j = 1,2,\dots, N$. To account for the periodic boundary conditions, we periodically extend functions defined on the nodes, i.e., $f_{i,j} = f_{i\pm N,j} = f_{i, j\pm N}$. The initial pair $(m_0, \phi_0)$ is then approximated by
\[ m^0_{i,j} := \frac{1}{\vert \Lambda_{i,j}\vert}\int_{\Lambda_{i,j}} m_0(x,y) \diff x\diff y\quad \text{ and } \quad \phi^0_{i,j} :=\frac{1}{\vert \Lambda_{i,j}\vert}\int_{\Lambda_{i,j}} \phi_0(x,y) \diff x\diff y \] and define the fully explicit scheme:

\begin{equation}\label{Eq_Scheme1}
\left \lbrace
\begin{split}
    \frac{m^{k+1}_{i,j}-m^k_{i,j}}{\Delta t} &= \frac{1}{\Delta x^2}D^2_i[m^k_{i,j}] +\frac{1}{\Delta y^2}D^2_j[m^k_{i,j}]  \nonumber\\
    &- \frac{\beta}{\Delta x} D^1_i[(\phi^k_{i,j} -(m^k_{i,j})^2) \Tilde{J}^k_{x,i,j}] - \frac{\beta}{\Delta y} D^1_j[(\phi^k_{i,j} -(m^k_{i,j})^2) \Tilde{J}^k_{y,i,j} ] \\
    \frac{\phi^{k+1}_{i,j}- \phi^k_{i,j}}{\Delta t} &= \frac{1}{\Delta x^2}D^2_i[\phi^k_{i,j}] +\frac{1}{\Delta y^2}D^2_j[\phi^k_{i,j}]  \nonumber\\
    &- \frac{\beta}{\Delta x} D^1_i[m^k_{i,j}(1-\phi^k_{i,j}) \Tilde{J}^k_{x,i,j}] - \frac{\beta}{\Delta y} D^1_j[m^k_{i,j}(1-\phi^k_{i,j}) \Tilde{J}^k_{y,i,j} ] + F(\phi_{i,j}^k)
\end{split} \right. ,
\end{equation}
where
\[D^2_l[f_l] := f_{l+1}-2f_{l} +f_{l-1}, \quad D^1_l[f_l] := f_{l+1} - f_{l-1},   \]
and $\Tilde{J}^k_{x,i,j}$ (respectively, $\Tilde{J}^k_{y,i,j}$) denotes the approximation of $\p_x J * m^k (x_i,y_j)$ (respectively, $\p_y J * m^k (x_i,y_j)$) using a fast Fourier transform method similar to \cite{TiwariKumaretal_2021_FastAccurateApproximation}. We remark that similar estimates to those found in Section \ref{Sec:AuxProblem} may be repurposed to discuss convergence of the finite volume scheme. However, we leave this to be rigorously shown in a future manuscript. A numerical demonstration of convergence for the scheme with $F(\phi) \equiv 0 $ can be found in \cite{LyonsMunteanetal2023}.

As in \cite{LyonsMunteanetal2023}, we observe morphology formation during the simulation. However, whereas in the previous results the shape-type (e.g.,`ball-like' or `bi-continuous') of morphologies stays consistent throughout the simulation, we observe a change in shape-type while there is an evaporation component. This phenomenon is demonstrated in Figure \ref{fig:M_Phi_Evap} where `ball-like' structures can be seen at early times ($t=1$) and they slowly grow as solvent evaporates into long continuous structures at later times. A similar situation pointing out a competition between phase separation and evaporation is reported in  \cite{Benoit_2021}. Here the shapes of the morphologies formed in a partially miscible mixture change during the evaporation process, while one notes that the spinodal instability occurs when the evaporation is fast.

\begin{figure}[h]
    \centering
    \includegraphics[scale = 0.6]{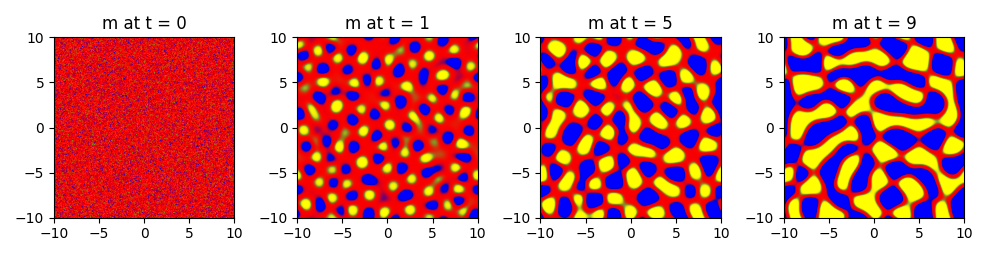}
    \includegraphics[scale = 0.6]{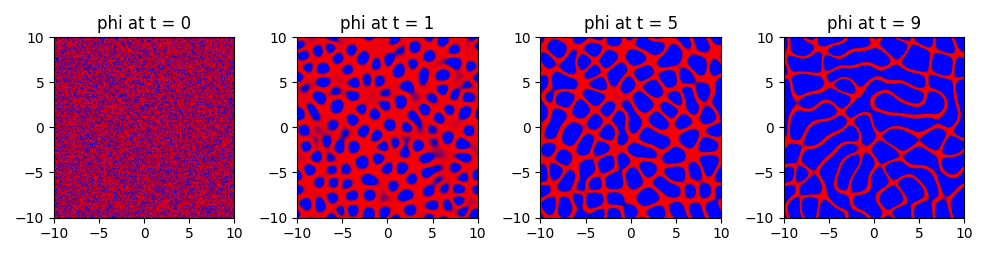}
    \caption{Simulation of equation \eqref{Eq:MainModel} with $\alpha = 0.1$, $\beta = 10$, and initial solvent ratio of 80\%. Regions where $m$ is positive are colored blue, negative regions are colored yellow, and regions $m$ is near zero are colored red. Similarly, regions where $\phi$ is near one are colored blue whereas regions near zero are colored red. }
    \label{fig:M_Phi_Evap}
\end{figure}

In Figure \ref{fig:Ratio_Calculations}, we present multiple measurements of the simulation. First, we present the ratio of solvent during the simulation over time given by $\|1-\phi \|_1$. As expected, we observe a decrease in solvent ratio over time. Next, we plot the $\| \cdot \|_1$ norms of both $m$ and $\phi$ (scaled by the norm of the initial condition) over time. The observed behavior is an increase in total volume of the solute. This is reminiscent of the Monte Carlo simulation of somewhat related microscopic models (see \cite{Andrea_PhysRevE,Andrea_EPJ}), where the evaporating solvent was replaced by solute particles. We note that the sudden drop in the $L^1$ norm of $m$ is due to the initial mixing of the solution (ternary mixture). In other words, this drop happens during the time period where morphologies are not yet well formed. In the aforementioned cited papers, the evaporating solvent was replaced by solute in such a way that the ratio of $+1$ and $-1$ particles was kept constant. We measure this effect in the third plot of Figure \ref{fig:Ratio_Calculations} where we plot the evolution of the ratio of $\|m^+\|_1$ and $\|m^-\|_1$.  Here, we observe an approximate 4\% variance of this ratio due to the numerical scheme. However, it appears to self correct over time. 

\begin{figure}[h]
    \centering
    \includegraphics[scale = 0.4]{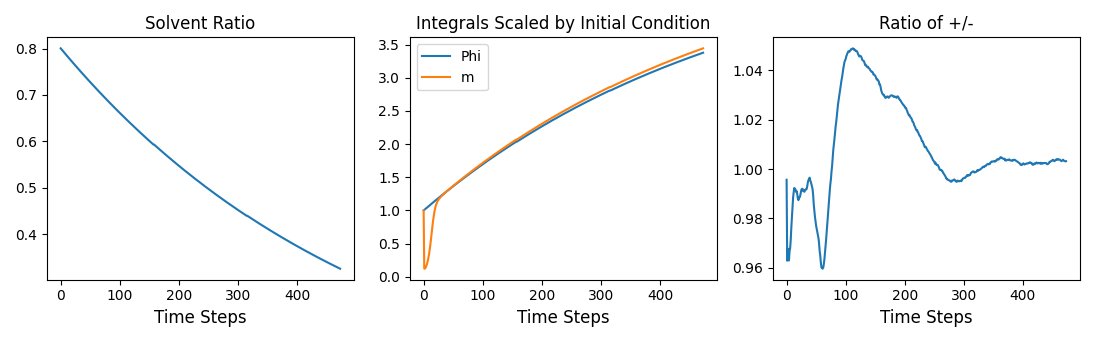}
    \caption{Solvent ratio, $L^1$ norms, and solute ratio of the simulation shown in Figure \ref{fig:M_Phi_Evap} plotted over time. The horizontal axis here represents the time step in the numerical scheme in thousands.}
    \label{fig:Ratio_Calculations}
\end{figure}

\section{Discussion and outlook}\label{Sec:Discussion}

Summarizing our work here, we have shown the existence of a unique bounded solution to system \eqref{Eq:MainModel} with standard parabolic regularity for some finite time. Moreover, the solutions to this system conserve the physical inequality \eqref{Ineq:MPhi}. We have included an example of how to adjust system \eqref{Eq:MainModel} to include a `from the top' evaporation process and included simulations of a simple example of this process. While the particular example of evaporation is simple, clear transitions between the patterns studied in \cite{LyonsMunteanetal2023} can be observed in Figure \ref{fig:M_Phi_Evap}.

The study of such type of phase separation models is rich with questions and it remains to be seen how the analysis of similar models such as the Cahn--Hilliard and Allen--Cahn equations compare to system \eqref{Eq:MainModel}. For instance, the study of the sharp interface via formal asymptotic expansion as done in \cite{giacomin1998phase,Pego,HennessyWagner2015,barua2023sharp} could lead to more insight on how the morphologies grow in time. One could add further complications to this study by applying the asymptotic expansion to system \eqref{Eq:MainModel} including the evaporation process. Another question relating to the analysis of the model is with respect to the strict separation property studied in \cite{gal2017nonlocal,gal2014longtime,gal2022separation}. Numerical simulations suggest that $||m| - \phi| \longrightarrow 0$ almost everywhere as $t \longrightarrow \infty$ for initial conditions satisfying assumption \eqref{Assump:MPhi}. This property has the physical meaning that the mixture asymptotically becomes well separated and has yet to be rigorously shown.  

Our main interest is to use the morphologies generated by system \eqref{Eq:MainModel} as a type of porous domain to model and simulate charge transport under some accepted physical assumptions \cite{Jenny,Khoa}. Such simulations can allow a quantitative study of the effectiveness of the patterns of observed morphologies. With this in mind, more simulations of system \eqref{Eq:MainModel} in three dimensions taking both perspectives of evaporation into account are crucial to best representing the physical system. Numerical schemes taking into account the gradient flow structure of the model presented in \cite{Marra} similar to those studied in \cite{Guan_Wang_Wise_2014,Guan_Lowengrub_Wang_Wise_2014} seem readily adaptable to system \eqref{Eq:MainModel} without evaporation. As future work, we would like to study ways to include the evaporation process inside of these schemes.    

\section*{Acknowledgments} The authors are grateful to the anonymous reviewers for their comments which significantly contributed to the quality of the manuscript. The authors also thank M. Eden (Karlstad, Sweden) for fruitful discussions on the mathematical analysis of the auxiliary problem and  S. A. Muntean (Karlstad) for brainstorming about the modeling of the evaporation production term. RL and AM acknowledge the financial support of Carl Tryggers Stieftelse via the grant CTS 21:1656.   
\bibliographystyle{plain}
\bibliography{Bib}

\begin{thebibliography}{10}

\bibitem{amann2011ordinary}
H.~Amann.
\newblock {\em Ordinary differential equations: an introduction to nonlinear
  analysis}, volume~13.
\newblock Walter de gruyter, 1990.

\bibitem{barua2023sharp}
A.~K. Barua, R.~Chew, S.~Li, J.~Lowengrub, A.~M{\"u}nch, and B.~Wagner.
\newblock Sharp-interface problem of the {O}hta-{K}awasaki model for symmetric
  diblock copolymers.
\newblock {\em Journal of Computational Physics}, 2023.

\bibitem{carrillo2014derivation}
J.~A. Carrillo, Y.-P. Choi, and M.~Hauray.
\newblock The derivation of swarming models: mean-field limit and {W}asserstein
  distances.
\newblock In A.~Muntean and F.~Toschi, editors, {\em Collective Dynamics from
  Bacteria to Crowds}, pages 1--46. CISM Series, 2014.

\bibitem{carrillo2020long}
J.~A. Carrillo, R.~S. Gvalani, G.~A. Pavliotis, and A.~Schlichting.
\newblock Long-time behaviour and phase transitions for the {McKean}--{V}lasov
  equation on the torus.
\newblock {\em Archive for Rational Mechanics and Analysis}, 235(1):635--690,
  2020.

\bibitem{CHAZELLE2017365}
B.~Chazelle, Q.~Jiu, Q.~Li, and C.~Wang.
\newblock Well-posedness of the limiting equation of a noisy consensus model in
  opinion dynamics.
\newblock {\em Journal of Differential Equations}, 263(1):365--397, 2017.

\bibitem{Andrea_EPJ}
E.~N.~M. Cirillo, M.~Colangeli, E.~Moons, A.~Muntean, S.~A. Muntean, and
  J.~{van Stam}.
\newblock A lattice model approach to the morphology formation from ternary
  mixtures during the evaporation of one component.
\newblock {\em Eur. Phys. J. Spec. Top.}, 228:55--68, 2019.

\bibitem{creton2016rubber}
C.~Creton and M.~Ciccotti.
\newblock Fracture and adhesion of soft materials.
\newblock {\em Reports on Progress in Physics}, 79(4):046601, 2016.

\bibitem{cummings2018modeling}
J.~Cummings, J.~S. Lowengrub, B.~G. Sumpter, S.~M. Wise, and R.~Kumar.
\newblock Modeling solvent evaporation during thin film formation in phase
  separating polymer mixtures.
\newblock {\em Soft Matter}, 14(10):1833--1846, 2018.

\bibitem{eden2022multiscale}
M.~Eden, C.~Nikolopoulos, and A.~Muntean.
\newblock A multiscale quasilinear system for colloids deposition in porous
  media: {W}eak solvability and numerical simulation of a near-clogging
  scenario.
\newblock {\em Nonlinear Analysis: Real World Applications}, 63:103408, 2022.

\bibitem{Evans_pde}
L.~C. Evans.
\newblock {\em Partial {D}ifferential {E}quations}, volume~19.
\newblock American Mathematical Society, $2^{nd}$ edition, 2010.

\bibitem{gal2017nonlocal}
C.~G. Gal, A.~Giorgini, and M.~Grasselli.
\newblock The nonlocal {C}ahn--{H}illiard equation with singular potential:
  well-posedness, regularity and strict separation property.
\newblock {\em Journal of Differential Equations}, 263(9):5253--5297, 2017.

\bibitem{gal2022separation}
C.~G. Gal, A.~Giorgini, and M.~Grasselli.
\newblock The separation property for 2{D} {C}ahn-{H}illiard equations: local,
  nonlocal and fractional energy cases.
\newblock {\em Discrete and Continuous Dynamical Systems}, 10, 2022.

\bibitem{gal2014longtime}
C.~G. Gal and M.~Grasselli.
\newblock Longtime behavior of nonlocal {C}ahn-{H}illiard equations.
\newblock {\em Discrete \& Continuous Dynamical Systems}, 34(1):145, 2014.

\bibitem{Giacomin}
G.~Giacomin and J.~L. Lebowitz.
\newblock Phase segregation dynamics in particle systems with long range
  interaction {I}. {M}acroscopic limits.
\newblock {\em J. Statist. Phys.}, 87:37--61, 1997.

\bibitem{giacomin1998phase}
G.~Giacomin and J.~L. Lebowitz.
\newblock Phase segregation dynamics in particle systems with long range
  interactions {II}: {I}nterface motion.
\newblock {\em SIAM Journal on Applied Mathematics}, 58(6):1707--1729, 1998.

\bibitem{Guan_Lowengrub_Wang_Wise_2014}
Z.~Guan, J.~S. Lowengrub, C.~Wang, and S.~M. Wise.
\newblock Second order convex splitting schemes for periodic nonlocal
  {C}ahn–{H}illiard and {A}llen–{C}ahn equations.
\newblock {\em Journal of Computational Physics}, 277:48–71, Nov 2014.

\bibitem{Guan_Wang_Wise_2014}
Z.~Guan, C.~Wang, and S.~M. Wise.
\newblock A convergent convex splitting scheme for the periodic nonlocal
  {C}ahn-{H}illiard equation.
\newblock {\em Numerische Mathematik}, 128(2):377–406, Oct 2014.

\bibitem{HennessyWagner2015}
M.~G. Hennessy, V.~M. Burlakov, A.~Goriely, B.~Wagner, and A.~M\"{u}nch.
\newblock Controlled topological transitions in thin-film phase separation.
\newblock {\em SIAM Journal on Applied Mathematics}, 75(1):38--60, 2015.

\bibitem{hoppe2004organic}
H.~Hoppe and N.~S. Sariciftci.
\newblock Organic solar cells: {A}n overview.
\newblock {\em Journal of Materials Research}, 19(7):1924--1945, 2004.

\bibitem{Khoa}
V.~A. Khoa and A.~Muntean.
\newblock Corrector homogenization estimates for a non-stationary
  {S}tokes-{N}ernst-{P}lanck-{P}oisson system in perforated domains.
\newblock {\em Communications in Mathematical Sciences}, 17(3):705--738, 2019.

\bibitem{lasry2007mean}
J.-M. Lasry and P.-L. Lions.
\newblock Mean field games.
\newblock {\em Japanese Journal of Mathematics}, 2(1):229--260, 2007.

\bibitem{liu2021long}
W.~Liu, L.~Wu, and C.~Zhang.
\newblock Long-time behaviors of mean-field interacting particle systems
  related to {McKean}--{V}lasov equations.
\newblock {\em Communications in Mathematical Physics}, 387(1):179--214, 2021.

\bibitem{LyonsMunteanetal2023}
R.~Lyons, S.~A. Muntean, E.~N.~M. Cirillo, and A.~Muntean.
\newblock A continuum model for morphology formation from interacting ternary
  mixtures: Simulation study of the formation and growth of patterns.
\newblock {\em Physica D: Nonlinear Phenomena}, 453:133832, 2023.

\bibitem{Marra}
R.~Marra and M.~Mourragui.
\newblock Phase segregation dynamics for the {B}lume–{C}apel model with {K}ac
  interaction.
\newblock {\em Stochastic Processes and their Applications}, 88(1):79--124,
  2000.

\bibitem{Miranville}
A.~Miranville.
\newblock {\em The {C}ahn-{H}illiard {E}quation: {R}ecent {A}dvances and
  {A}pplications}.
\newblock CBMS-NSF Regional Conference Series in Applied Mathematics. SIAM,
  2019.

\bibitem{Muller_etal2022}
M.~M\"{u}ller, A.~Lang, M.~Kl\"{u}ppel, and U.~Giese.
\newblock Influence of phase morphology on viscoelastic properties of rubber
  blends.
\newblock In C.~Marano, F.~B. Vangosa, L.~Andena, and R.~Frassine, editors,
  {\em Constitutive Models for Rubber XII: Proceedings of the 12th European
  Conference on Constitutive Models for Rubber (ECCMR 2022)}. CRC Press, 2022.

\bibitem{AdrianContinuum}
A.~Muntean.
\newblock {\em Continuum {M}odeling: {A}n {A}pproach {T}hrough {P}ractical
  {E}xamples}.
\newblock SpringerBriefs in Mathematical Methods. Springer International
  Publishing, 2015.

\bibitem{Andrea_PhysRevE}
S.~A. Muntean, V.~C.~E. Kronberg, M.~Colangeli, A.~Muntean, J.~{van Stam},
  E.~Moons, and E.~N.~M. Cirillo.
\newblock Quantitative analysis of phase formation and growth in ternary
  mixtures upon evaporation of one component.
\newblock {\em Phys. Rev. E}, 106:025306, 2022.

\bibitem{Jenny}
J.~Nelson, J.~J. Kwiatkowski, J.~Kirkpatrick, and J.~M. Frost.
\newblock Modeling charge transport in organic photovoltaic materials.
\newblock {\em Acc. Chem. Res.}, 42(11):1768--1778, 2009.

\bibitem{Pego}
R.~L. Pego.
\newblock Front migration in the nonlinear {C}ahn-{H}illiard equation.
\newblock {\em Proceedings of the Royal Society of London. Series A,
  Mathematical and Physical Sciences}, 422(1863):261--278, 1989.

\bibitem{Presutti}
E.~Presutti.
\newblock {\em Scaling {L}imits in {S}tatistical {M}echanics and
  {M}icrostructures in {C}ontinuum {M}echanics}.
\newblock Theoretical and Mathematical Physics. Springer, 2008.

\bibitem{Benoit_2021}
R.~Rabani, H.~Sadafi, H.~Machrafi, M.~Abbasi, B.~Haut, and P.~Dauby.
\newblock Influence of evaporation on the morphology of a thin film of a
  partially miscible binary mixture.
\newblock {\em Colloids and Surfaces A: Physicochemical and Engineering
  Aspects}, 612:126001, 2021.

\bibitem{Schaefer}
C.~Schaefer.
\newblock {\em Theory of nanostructuring in solvent-deposited thin polymer
  films}.
\newblock PhD thesis, Technische Universiteit Eindhoven, 2016.

\bibitem{Mario}
M.~Setta, V.~C.~E. Kronberg, S.~A. Muntean, E.~Moons, J.~{van Stam}, E.~N.~M.
  Cirillo, M.~Colangeli, and A.~Muntean.
\newblock A mesoscopic lattice model for morphology formation in ternary
  mixtures with evaporation.
\newblock {\em Communications in Nonlinear Science and Numerical Simulation},
  119:107083, 2023.

\bibitem{TiwariKumaretal_2021_FastAccurateApproximation}
A.~K. Tiwari, A.~Pandey, J.~Paul, and A.~Anand.
\newblock Fast accurate approximation of convolutions with weakly singular
  kernel and its applications, 2021.
\newblock (arXiv.2107.03958).

\bibitem{Vera_2017}
J.~M.~R. Vera.
\newblock A convergent iterative method for a logistic chemotactic system.
\newblock {\em Revista Colombiana de Matemáticas}, 51(1):103--117, 2017.

\end{thebibliography}

\end{document}